\documentclass{article}
\usepackage[utf8]{inputenc}
\usepackage{xcolor}

\usepackage{amsmath}
\usepackage[pdftex]{graphicx}
\usepackage{amsfonts}
\usepackage{amssymb}
\usepackage{wrapfig}
\usepackage{setspace}
\usepackage{epsfig}
\usepackage{color}
\usepackage{tikz}
\usepackage{cancel}
\usepackage{float}
\usepackage{url}
\usepackage{hyperref}
\usepackage{authblk}
\usepackage{fullpage}

\usepackage{amsthm}
\usepackage{wasysym}
\usepackage{tkz-graph}
\usepackage{tikzscale}
\usepackage{subcaption}
\usepackage{mathtools}
\usepackage{todonotes}

\usetikzlibrary{positioning}
\usetikzlibrary{cd}

\theoremstyle{definition}
\newtheorem{definition}{Definition}[section]

\newtheorem{theorem}{Theorem}[section]
\newtheorem{corollary}{Corollary}[theorem]
\newtheorem{lemma}[theorem]{Lemma}
\newtheorem{prop}[theorem]{Proposition}

\theoremstyle{plain}

\newtheoremstyle{case}{}{}{}{}{}{:}{ }{}
\theoremstyle{case}

\newcommand{\N}{\mathbb{N}}

\newcommand{\Z}{\mathbb{Z}}

\newcommand{\Prob}{\mathbb{P}}

\newcommand{\gn}{\text{gn}}
\newcommand{\ex}{\text{ex}}
\newcommand{\sat}{\text{sat}}
\newcommand{\Ex}{\text{Ex}}
\newcommand{\Sat}{\text{Sat}}
\newcommand{\fix}{\text{fix}}
\newcommand{\cp}{\text{cp}}

\title{Guessing Numbers and Extremal Graph Theory}
\author{Jo Martin and Puck Rombach}
\affil{\footnotesize{Department of Mathematics \& Statistics, University of Vermont, VT, USA.}}
\date{\today}

\begin{document}

\maketitle

\begin{abstract}
For a given number of colors, $s$, the guessing number of a graph is the (base $s$) logarithm of the cardinality of the largest family of colorings of the vertex set of the graph such that the color of each vertex can be determined from the colors of the vertices in its neighborhood. This quantity is related to problems in network coding, circuit complexity and graph entropy. We study the guessing number of graphs as a graph property in the context of classic extremal questions, and its relationship to the forbidden subgraph property. We find the extremal number with respect to the property of having guessing number $\leq a$, for fixed $a$. Furthermore, we find an upper bound on the saturation number for this property, and a method to construct further saturated graphs that lie between these two extremes. We show that, for a fixed number of colors, bounding the guessing number is equivalent to forbidding a finite set of subgraphs.
\end{abstract}

\section{Introduction}
The guessing number of a graph is a graph invariant introduced by S\o ren Riis, as a tool to work on problems in network coding \cite{riis_utilising_2005} and circuit complexity \cite{riis_information_2006}. It is one of many other variants of multiplayer information games, such as the hat guessing game, Ebert's game, hats-on-a-line game. For a review, see~\cite{butler_hat_2009}. We will give a formal definition of the guessing game in Section~\ref{sec:defs}. Informally, imagine that $n$ players are positioned on the vertices of an undirected graph $G$. Two players can see each other if their vertices share an edge in $G$. Each player is assigned a hat with a color chosen uniformly from a set of $s$ colors, independently of other players. The players guess the color of their own hats simultaneously, where the goal is to maximize the probability that all players guess correctly. Players cannot see their own hats. Instead, they base their guess on a previously agreed upon strategy and the colors of the other players' hats that they can see. The guessing number $\gn (G)$ reflects the quality of a best possible guessing strategy on the graph $G$.

Riis proved that computing this particular guessing number of a graph is equivalent to solving an information flow problem on an associated network \cite{riis_utilising_2005}. In particular, this refers to the solvability of the multiple unicast coding problem~\cite{ahlswede2000network}. It is also related to the problem of index coding with side information~\cite{alon2008broadcasting,gadouleau_graph-theoretical_2011}. Gadouleau showed that this problem can also be recast in terms of fixed points of finite dynamical systems~\cite{gadouleau2018possible}. Christofides and Markstr\"{o}m were the first to expand the study of guessing numbers of undirected graphs \cite{christofides_guessing_2011}. They found the exact guessing numbers of a class of graphs that contains the perfect graphs, namely the graphs whose independence number equals the clique cover number of their complements. The guessing numbers of undirected triangle free graphs \cite{cameron_guessing_2016} and odd cycles \cite{atkins_guessing_2017} have also been studied, but very few other graphs have known guessing number. 

Extremal graph theory is a well-studied area of graph theory that concerns itself with how large (or small) a graph can be while fulfilling certain properties. The traditional T\'{u}ran problem asks how many edges a graph on $n$ vertices can have, while avoiding a subgraph isomorphic to some $F$, or to any $F$ in a given family $\mathcal{F}$. This type of question was introduced by Mantel~\cite{mantel_problem_1907}, and solved by T\'{u}ran for complete graphs~\cite{turan1941extremal}. This is a fundamental question in combinatorics which has been studied extensively since then. For a survey, see for example~\cite{sidorenko1995we}. Similarly, the original saturation problem asks the question of how few edges an $F$-free graph on $n$ vertices can have while having the property that the addition of any edge creates gives rise to a subgraph $F$, or to any $F$ in a given family $\mathcal{F}$. The saturation number was introduced by Erd\H{o}s, Hajnal and Moon in~\cite{erdos1964problem}. For a survey, see for example~\cite{faudree2011survey}. 

In this paper we look at extremal and saturation questions of the guessing number. We define, in terms that parallel prior extremal work on subgraphs, extremal and saturation numbers for the guessing number, and then determine the extremal number as well as a constant bound for the saturation number. These questions are of interest, especially when we think of the guessing number as it relates to the efficiency of a network in terms of its ability to transmit a message versus the number of links that are used. The graph property of having guessing number at least $a$ is equivalent to the property of avoiding a finite family of subgraphs $\mathcal{F}_{a}$.

The remainder of this paper is organized as follows. In section~\ref{sec:defs}, we give formal definitions related to guessing numbers, strategies, and extremal and saturation numbers, as well as a few useful lemmas. In Section~\ref{sec:ext}, we present the extremal number for graphs of bounded guessing number. This result does not depend on the number of colors used. In Section~\ref{sec:sat}, we provide a constant upper bound on the saturation number (that only depends on the guessing number $a$, not on $n$) that applies to any number of colors. The saturation number may depend on the number of colors used, unlike the extremal numbers. In Section~\ref{sec:spect}, we discuss a method of building further saturated graphs. In Section~\ref{sec:Fgneq}, we look further into the relationship between the bounded guessing number property and forbidden subgraphs, and show that for a fixed number of colors, bounding the guessing number is equivalent to forbidding a finite set of subgraphs.

\section{Definitions and useful results}\label{sec:defs}

This section is split into three subsections, which deal with guessing numbers, saturation and extremal numbers, and with graph entropy, respectively.

\subsection{Guessing Numbers}
In 2006, S\o ren Riis introduced a new guessing game variant played on directed graphs. This guessing game was originally developed by Riis and Mikkel Thorup in 1997 \cite{riis_utilising_2005}. Similarly to some of the other games, players are assigned hat colors at random, can decide on a strategy beforehand but cannot communicate after the hats have been assigned, and all guess simultaneously. Riis introduces a new win condition: The players are trying to maximize the probability that \textit{every} player guesses correctly. 

We will consider only undirected graphs. A graph $G$ is a pair $(V,E)$, where $V=\{ v_i\}_1^n$ is a set of vertices, and $E \subseteq \binom{V}{2}$ a set of edges. We will use the convention that $n=|V(G)|$ and $m=|E(G)|$. We let $N(v)$ denote the \emph{neighborhood} of a vertex $v$, \emph{i.e.} $N(v)=\{w \in V(G):\; {v,w} \in E(G) \}$. We let $N[v]=N(v)\cup \{ v \}$ denote the closed neighborhood of $v$. We use $K_n$ to denote the complete graph on $n$ vertices, and $E_n$ the empty graph on $n$ vertices. We use the symbol $\oplus$ to denote the operator that forms the join of two graphs, and $+$ to denote the disjoint union.

In Riis's guessing game, every vertex $G$ is assigned a color from a color set $[s]$, uniformly at random and independently of other vertices. Each vertex guesses the color that has been assigned to it, based on the information of the colors assigned to its neighbors. The collection of $n$ guessing functions, one for each vertex, is called a strategy or protocol for the guessing game for $G$ with $s$ colors. The goal of the guessing game is to find a protocol that maximises the probability that every vertex guesses its own color.

In this paper we use the following set of definitions related to guessing games and guessing numbers of undirected graphs. These definitions have been slightly modified from Riis's orginal presentation.

\begin{definition}
A \textit{protocol} or \textit{strategy} for graph $G$ with respect to a color set of size $s$ is a set of functions, $\mathcal{P}=\{f_i\}_1^n$ where each $f_i$ is a function $f_i: \Z_s^n \rightarrow \Z_s$ associated with a vertex $v_i\in V(G)$, where $f_i$ may only depend on the colors of the verices in $N(v_i)$. Then we can think of the protocol itself as a function $\mathcal{P}:\Z_s^n\rightarrow \Z_s^n$. 
\end{definition}

We then use the following definition of the guessing number, styled after Christofides and Markström \cite{christofides_guessing_2011}.
For $s\in\N$ and $G$ a graph, we let the \textit{guessing number} be given as $\gn(G,s)=k$, where $k$ is the largest value such that there exists a protocol $\mathcal{P}$ where every vertex guesses its own value with probability $s^{k-n}$.
A more compact definition in terms of the fixed points of a protocol was first introduced Wu, Cameron, and Riis in 2009 \cite{wu_guessing_2009}. A protocol defined above as function $\mathcal{P}:\Z_s^n \rightarrow \Z_s^n$ guesses correctly whenever $\mathcal{P}(c)=c$, or when a coloring $c$ is a fixed point of $\mathcal{P}$. Such colorings are those for which the strategy is successful. This allows us to define the guessing number in terms of the fixed points of a strategy.

\begin{definition}
\label{def:gn}
The \textit{guessing number} of a graph $G$ with respect to an $s$-guessing game is
\begin{equation*}
    \gn(G,s) = \log_s \max_{\mathcal{P}}\{\text{fix}(\mathcal{P})\},
\end{equation*}
where $\text{fix}(\mathcal{P})$ is the number of fixed points of a strategy $\mathcal{P}$.
\end{definition}

\begin{definition}
\label{def:ggn}
The \textit{general guessing number} of a graph $G$ is
\begin{equation*}
    \gn(G) = \sup_s \gn(G,s).
\end{equation*}
\end{definition}

Much of the foundational work purely on the guessing number was done by Christofides and Markstr\"{o}m in 2011. Their initial bounds and exposition on some of the fundamentals of the guessing number of undirected graphs are indispensable for this paper. In their 2001 paper, Christofides and Markstr\"{o}m present a general upper and lower bounds for the guessing number using the clique cover number and independence number of graphs. We present Lemmas~\ref{lem:subgraphadd} and~\ref{lem:kngn}, which together imply the lower bound in Lemma~\ref{lem:cpalpha}.
\begin{lemma} \cite{wu_guessing_2009,christofides_guessing_2011}
\label{lem:subgraphadd}
For two disjoint subgraphs, $H_1,H_2\subseteq G$, we have  \[\gn(G,s)\geq \gn(H_1,s)+\gn(H_2,s).\]
\end{lemma}
\begin{lemma} \cite{christofides_guessing_2011}
\label{lem:kngn}
For the complete graph $K_n$, we have
\[ \gn(K_n,s)=n-1. \]
\end{lemma}
We let $\alpha(G)$ be the independence number of $G$, which is the cardinality of a largest independent set in $G$. We let $\cp (G)$ be the clique decomposition number of $G$, which is the fewest number of classes in a partition of $G$ such that each class induces a clique. This is equal to the chromatic number of the complement of the graph $G$.
\begin{lemma} \cite{christofides_guessing_2011}
\label{lem:cpalpha}
For every graph $G$ on $n$ vertices, 
\begin{equation*}
   n-\cp(G)\leq \gn(G,s)\leq n-\alpha(G).
\end{equation*}
\end{lemma}

\subsection{Extremal and saturation numbers}
We first define the extremal and saturation number in their traditional forms, in terms of forbidden subgraphs. Let $\mathcal{F}$ be a family of graphs. We say that a graph $G$ is $\mathcal{F}$-saturated if $G$ does not contain any graph $F \in \mathcal{F} $ as a subgraph, but for any $e\in \overline{E(G)}$ we have that $G+e$ does contain a subgraph $F \in \mathcal{F} $. If $G$ is $\mathcal{F}$-saturated with $\mathcal{F}=\{ F\}$, we say that $G$ is $F$-saturated.

\begin{definition}\label{def:exF}
The extremal number $\ex (n,\mathcal{F})$ (resp.,  $\ex (n,F)$) is the maximum number of edges over all graphs on $n$ vertices that are $\mathcal{F}$-free (resp., $F$-free). The family of such graphs on $n$ vertices and the extremal number of edges is denoted by $\Ex(n,\mathcal{F})$ (resp.,  $\Ex (n,F)$).
\end{definition}
Note that all graphs in $\Ex(n,\mathcal{F})$ (resp.,  $\Ex (n,F)$) must be  $\mathcal{F}$-saturated (resp., $F$-saturated). 

\begin{definition}\label{def:satF}
The saturation number $\sat (n,\mathcal{F})$ (resp.,  $\sat (n,F)$) is the minimum number of edges over all graphs on $n$ vertices that are $\mathcal{F}$-saturated (resp., $F$-saturated). The family of graphs on $n$ vertices and the saturation number of edges is denoted by $\Sat(n,\mathcal{F})$ (resp.,  $\Sat (n,F)$).
\end{definition}

Similarly, we say that a graph $G$ is $(\gn_s\geq a)$-saturated if $\gn_s(G)< a$ and for any $e\in \overline{E(G)}$ we have that $gn_s(G+e)\geq a$. We then, in a logical way, define the extremal and saturation values of a guessing number in the spirit of traditional definitions with respect to forbidden subgraphs.

\begin{definition}\label{def:exgn}
The extremal number $\ex(n,\gn_s \geq a)$ is the maximum number of edges over all graphs on $n$ vertices that have guessing number $< a$. The family of such graphs on $n$ vertices and the extremal number of edges is denoted by $\Ex(n,\gn_s\geq a)$.
\end{definition}
Note that all graphs in the family $\Ex(n,\gn_s\geq a)$ are $(\gn_s\geq a)$-saturated.

\begin{definition}\label{def:satgn}
The saturation number $\sat(n,\gn_s \geq a)$ is the minimum number of edges over all graphs on $n$ vertices that are $(\gn_s\geq a)$-saturated. The family of such graphs on $n$ vertices and the saturation number of edges is denoted by $\Sat(n,\gn_s\geq a)$.
\end{definition}

The graph properties of attaining a given guessing number and containing a subgraph from a given family of graphs are strongly related. In Lemma~\ref{lem:Fgneq}, we will show that for every $s \in \mathbb{N}$, $a \in \mathbb{R}$, there exists a unique finite family of minimal forbidden subgraphs $\mathcal{F}_{s,a}$ such that, for any graph $G$,
\[ \gn_s(G)< a \;\; \Leftrightarrow \;\; G \mbox{ is }\mathcal{F}_{s,a}\mbox{-free}. \] 

 Given this fact, the reader might wonder why we need separate definitions for the extremal and saturation numbers with respect to forbidden subgraphs and guessing numbers, respectively. One reason is that, although we know that a graph family $\mathcal{F}_{s,a}$ exists, in most cases we do not know what this family is. Another reason is that we do not have a monotonicity between these two properties. For example, when $a\leq b$, we need not have $\mathcal{F}_{s,a} \subseteq \mathcal{F}_{s,b}$, or that every graph in $\mathcal{F}_{s,a}$ is contained in some graph in $\mathcal{F}_{s,b}$. 

\begin{lemma}\label{lem:monFgn}
 If $a\leq b$, then for every graph $F_b \in \mathcal{F}_{s,b} $ there exists an $F_a \in \mathcal{F}_{s,a}$ such that $F_a \subseteq F_b$.
\end{lemma}
\begin{proof}
For the sake of contradiction, suppose that $a \leq b$, and that there exists an $F_b\in \mathcal{F}_{s,b} $ such that no graph in $\mathcal{F}_{s,a}$ is a subgraph of $F_b$. This implies that $\gn_s (F_b) \geq b$, by the definition of $\mathcal{F}_{s,b}$, but also that $\gn_s (F_b) < a$, since it is $\mathcal{F}_{s,a}$-free. This is a contradiction.
\end{proof}

We will discuss a few aspects of the behavior of the extremal and saturation functions which are also seen in the well-studied setting of forbidden subgraphs. The following properties of the extremal 
\begin{lemma}\label{lem:monF}
For every $F' \subseteq F$ and every $\mathcal{F}' \subseteq \mathcal{F}$,
\begin{itemize}
    \item[(i)] $\ex (n,F') \leq \ex (n,F)$,
    \item[(ii)] $\ex (n,\mathcal{F}') \leq \ex (n,\mathcal{F})$,
    \item[(iii)] $\ex (n,\mathcal{F}) \leq \ex (n+1,\mathcal{F})$.
\end{itemize}
\end{lemma}
In the case of saturation, none of the above monotonic behaviors hold necessarily.  See~\cite{kaszonyi_saturated_1986} or~\cite{faudree2011survey}, for a survey. For guessing numbers, we will discuss analogous types of monotonicity in Sections~\ref{sec:ext} and~\ref{sec:sat}.

\subsection{Graph Entropy}

In~\cite{riis_graph_2007}, Riis develops the concept of graph entropy and connects it to the guessing number. This result allows us to use entropy inequalities to calculate the guessing number of graphs. 

\begin{definition}
\label{def:entropy}
Let $\{X_i\}_1^n$ be a collection of random variables each taking values from the same finite set $A$. Then for some appropriately chosen base, $b$, the information entropy (or Shannon's entropy) of the collection $\{X_i\}_1^n$ is defined as
\begin{equation*}
    H(\{X_i\}_1^n) =-\sum_{\mathclap{\left(x_1,\ldots, x_n\right) \in A^n }} \Prob(X_1=x_1,\ldots X_n=x_n)\log_b \Prob(X_1=x_1,\ldots X_n=x_n)
\end{equation*}
\end{definition}

The following lemma summarizes basic properties of entropy. 

\begin{lemma}\cite{shannon_mathematical_1948}\label{lem:shannon}
For $X,Y$ random variables, we have
\[0 \le H(X) \le H(X,Y).\]
\end{lemma}

\begin{definition}
For a graph $G$ and positive integer $s$, let $\mathcal{P}$ be a nontrivial strategy. Let $X_\mathcal{P}=(X_1,X_2,\ldots, X_n)$ be a random variable representing picking a coloring uniformly at random from the fixed points of $\mathcal{P}$ where $X_i$ is the color of vertex $i$.
\end{definition}

\begin{lemma} \cite{riis_graph_2007}
\label{lem:entpop}
For a graph $G$ positive integer $s$, and strategy $\mathcal{P}$, consider $X_\mathcal{P}$ and an arbitrary subset $S=\{v_1,\dots,v_t\}\subseteq V(G)$, without loss of generality. If $N(v_1)\subseteq S$, then

\begin{equation*}
    H(\{X_i\}_1^t)=H(\{X_i\}_2^t).
\end{equation*}
\end{lemma}

We can link the special case when this random variable is picking from an optimal strategy to the guessing number. 

\begin{lemma}\label{lem:gnent}\cite{riis_graph_2007}
Let $\mathcal{P}$ be an optimal strategy on a graph $G$. Then
\begin{equation*}
   H(X_\mathcal{P})=\gn (G).
\end{equation*}
\end{lemma}

Let $\mathcal{P}$ be an (optimal) strategy on a graph $G$ for a given $s$. From the basic properties of entropy~\cite{riis_graph_2007}, we have 
\begin{lemma}\label{lem:entleq1}
 For any $X_i \in X_{\mathcal{F}}$, we have 
\[H(X_i)\leq 1.\]
\end{lemma}

\section{Extremal Numbers}\label{sec:ext}

In this section, we present a construction of the extremal graph with guessing number strictly less than $a$, and we prove that this construction is unique up to isomorphism, as is common with extremal graphs. Our extremal construction has integer guessing number. Any graph with more edges than the extremal construction has guessing number at least 1 more than the extremal construction. Therefore, we will see that  
\[ \ex(n,\gn \geq a)=\ex(n,\gn\geq \lceil a \rceil),\]
as well as 
\[ \Ex(n,\gn\geq a)=\Ex(n,\gn\geq \lceil a \rceil).\]

\begin{prop}
\label{prop:dsbound}
For any graph $G$, if $\chi(G)\geq k$, then $|E(G)|\geq \binom{k}{2}$ edges. If $\chi(G)=k$ and $|E(G)|= \binom{k}{2}$, then $G \sim K_k + E_{n-k}$.
\end{prop}
\begin{proof}
This follows from basic properties of the chromatic number. For example, see Theorem 5.2.1 and Lemma 5.2.3 in~\cite{reinhard_diestel_graph_2016}.
\end{proof}

By noting that $\chi(G)=\cp(\overline{G})$ and substituting $n-k$ for $k$ for later convenience, we state the following corollary of Proposition~\ref{prop:dsbound}.
\begin{corollary}\label{cor:cpbound}
For any graph $G$, if $\cp(G)\geq n-k$, then $|E(G)|\leq \binom{n}{2}- \binom{n-k}{2}$ edges. If $\cp(G)=n-k$ and $|E(G)|=\binom{n}{2}- \binom{k}{2}$, then $G \sim K_k \oplus E_{n-k}$.
\end{corollary}

Theorem~\ref{thm:extremal} gives a complete characterization of the extremal graphs and numbers for any guessing number. Note that if $a>n-1$, then the extremal graphs and numbers are undefined, as it is not possible for a simple grpah $G$ to have a guessing number greater than $n-1$.

\begin{theorem}
\label{thm:extremal}
For $n>\lceil a \rceil$, let $k=\lceil a \rceil-1$. Then, we have 
\[\Ex(n,gn(G(n))\geq a)=\{ K_{k} \oplus E_{n-k} \},\]
and
\[\ex(n,gn(G(n))\geq a)=\binom{n}{2}-\binom{n-k}{2}.\]
\end{theorem}
\begin{proof}
It is easy to see that $\alpha(K_{k} \oplus E_{n-k})=k$. Furthermore, we can find a partition of the vertices into $n-k$ cliques by placing each vertex of the independent set of size $n-k$ into one of the classes, and distributing the remaining vertices arbitrarily over the $n-k$ classes. We show an example of a graph $K_{k} \oplus E_{n-k}$ and a clique partition in Figure~\ref{fig:exexample}. 
Now, by Lemma~\ref{lem:cpalpha}, we see that 
\[ \gn(K_{k} \oplus E_{n-k})=k  .\]

Above we have an example of a graph on $n$ vertices, $\binom{n}{2}-\binom{n-k}{2}$ edges, and guessing number $k$. It follows from Corollary~\ref{cor:cpbound}, that any graph $G$ on $n$ vertices and at least $\binom{n}{2}-\binom{n-k}{2}+1$ edges must have $\gn (G)\geq k+1$. Also, by Corollary~\ref{cor:cpbound}, it follows that the graph $K_{k} \oplus E_{n-k}$ is the unique graph on $n$ vertices, $\binom{n}{2}-\binom{n-k}{2}$ edges, and guessing number $k$.

\end{proof}

Intuitively, adding edges to a graph (weakly) increases its guessing number. More information is never bad. However, this result is in a sense extremal for the number of ``useless" edges that the graph has in terms of an optimal strategy. If we look at Figure~\ref{fig:exexample}, we see that a strategy that forms guesses independently in each of the classes of the clique partition only uses $\binom{k+1}{2}$ edges, and ignores the remaining $k(n-k-1)$ edges. If $k$ is treated as a constant, this implies that only a constant number of edges out of a number that grows linearly with $n$ is ``useful". The number of edges in our graph is therefore quite uninformative if we do not know the structure. 
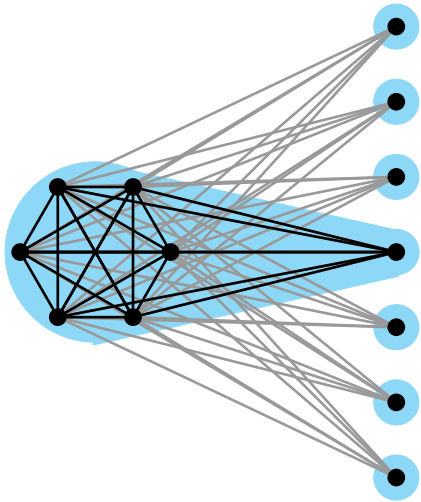
\begin{figure}[h]
\label{exexample}
\centering
  \begin{tikzpicture}[scale=1]
\draw[cyan!40,fill=cyan!40,line width=0] (0,0) circle (1.2);
\draw[cyan!40,fill=cyan!40,line width=2] (0,1.2) -- (0,-1.2) -- (4,-.3) -- (4,.3);
\draw[cyan!40,fill=cyan!40,line width=0] (4,0) circle (.3);
\foreach \i in {1,...,7}
 {
 \draw[cyan!40,fill=cyan!40,line width=0] (4,-4+\i) circle (.3);
  \foreach \j in {1,...,6}
  {
  \draw[black!40,line width=1pt] (4,-4+\i) -- (60*\j:1);
  }
  \draw[fill=black!100,line width=1] (4,-4+\i) circle (.1);
 }
 \foreach \i in {1,...,6}
{\begin{scope} [rotate=60*\i]
\draw[black!100,line width=1pt] (0:1) -- (60:1);
\draw[black!100,line width=1pt] (0:1) -- (120:1);
\draw[black!100,line width=1pt] (0:1) -- (180:1);
 \draw[fill=black!100,line width=1] (0:1) circle (.1);
\end{scope}}
\foreach \j in {1,...,6}
  {
  \draw[black!100,line width=1pt] (4,0) -- (60*\j:1);
  }
\end{tikzpicture}
\caption{Example of a graph $K_6 \oplus E_7$, with a clique partition (indicated by shaded areas) into 7 cliques.}
\label{fig:exexample}
\end{figure}

Lemma~\ref{lem:monF} shows the monotonic behavior of the extremal function in the case of forbidden subgraphs. Given the relationship between guessing numbers and forbidden subgraphs, only the analogue of Lemma~\ref{lem:monF}(iii) follows directly from the definitions. Together with Theorem~\ref{thm:extremal}, we obtain the following corollary.

\begin{corollary}\label{cor:monexgn}
When $a \leq b$, we have
\begin{itemize}
    \item[(i)] $\ex (n,\gn_s \geq a)\leq \ex (n,\gn_s \geq b) $,
     \item[(ii)] $\ex (n,\gn_s \geq a)\leq \ex (n+1,\gn_s \geq a) $.
\end{itemize}
\end{corollary}

\section{Saturation Numbers}\label{sec:sat}
We now move to the saturation number, which is the smallest number of edges a saturated graph can have. We begin with the saturation numbers for the properties of having guessing number at least $2$ and at least $n-1$. For these cases, can find exact numbers.

\begin{lemma}\label{lem:2sat}
For all $n$, $\sat(n,\gn \geq 2)=n-1$.
\end{lemma}

\begin{proof}
Suppose there is a graph $G$ on $n$ vertices with fewer than $n-1$ edges and $\gn(G)<2$. As a consequence of Lemmas~\ref{lem:subgraphadd} and~\ref{lem:kngn}, it cannot have a triangle or a matching of size 2. This implies that $G \sim K_{1,t}+(n-t-1)K_1$, for $0\leq t \leq n-1$. 

Suppose, for the sake of contradiction, that $G$ has an isolated vertex. Then, we can add an edge from an isolated vertex to the vertex at the center of the star $K_{1,t}$. Adding such an edge keeps the guessing number strictly below 2, since there is still an independent set of size $n-1$. Therefore, $G$ is not saturated with respect to $\gn(G)\geq 2$. We conclude that $G \sim K_{1,n-1}$. 
\end{proof}

Note that, by Theorem \ref{thm:extremal}, $\ex(n, \gn \geq 2)=n-1$. Therefore, we have 
\[\sat (n, \gn \geq 2)=\ex(n, \gn \geq 2)=n-1\] 
and we have 
\[ \Sat (n, \gn \geq 2)=\Ex(n, \gn \geq 2)= \{ K_{1,n} \} .\]

We find a similar complete picture for the extremal and saturation numbers (as well as the respective graphs) of $\gn \geq n-1$.

\begin{lemma}
For all $n$, $\sat(n,\gn(G(n))\geq n-1)=\binom{n}{2}-1$.
\end{lemma}

\begin{proof}
The only graph with guessing number $n-1$ is the complete graph on $n$ vertices. Therefore, the only graph saturated with respect to $\gn(G(n))\geq n-1$ has one less edge then the complete graph, or $\binom{n}{2}-1$ edges.
\end{proof}

We have 
\[\sat (n, \gn \geq n-1)=\ex(n, \gn \geq n-1)=\binom{n}{2}-1\] 
and we have 
\[ \Sat (n, \gn \geq n-1)=\Ex(n, \gn \geq n-1)= \{ K_{n-2}\oplus E_2 \} .\]

These two results give us that for guessing numbers $2$ and $n-1$, the saturation number is the same as the extremal number. However, for other guessing numbers there is a gap between the extremal number and the guessing number. Particularly, we show that for any constant integer guessing number $a>2$, the saturation number is bounded by a constant as $n$ grows.  

We begin by looking at saturation for guessing number $3$. We will generalize this construction later on, but this example will serve as a warm-up.  Consider the graph $C_5$. Christofides and Markström bound the guessing number for all $s$ with the following theorem.
\begin{theorem}\cite{christofides_guessing_2011}
\label{th:cyclebound}
For $s$ and $k$ positive integers, $\gn_s(C_{2k+1})\leq \frac{2k+1}{2}$.
\end{theorem}

\begin{lemma}
\label{lem:5cycle}
For $n\geq 5$ and all $s$, $\sat(n,\gn_s\geq 3)\leq 5$.
\end{lemma}
\begin{proof}
For $n\geq 5$, we let $G \sim C_5 + E_{n-5}$. As shown in Figure~\ref{fig:c5+edge}, the addition of any edge $e$ to this graph results in $G+e$ having a clique cover of size $n-3$. Note that in the figure, the three cases are shown without any additional isolated vertices drawn. The second and third case only apply when $n\geq 6$ and $n \geq 7$, respectively.

\end{proof}

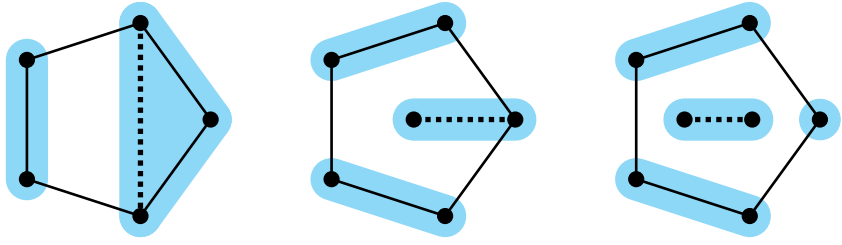
\begin{figure}[h]
\centering
  \begin{tikzpicture}[scale=.9]
\draw[cyan!40,fill=cyan!40,line width=16, line cap=round] (216:1.5) -- (144:1.5);
  \draw[cyan!40,fill=cyan!40,line width=16, line cap=round] (0:1.5) -- (72:1.5);
   \draw[cyan!40,fill=cyan!40,line width=16, line cap=round] (0:1.5) -- (-72:1.5);
    \draw[cyan!40,fill=cyan!40,line width=16, line cap=round] (-72:1.5) -- (72:1.5);
    \draw[cyan!40,fill=cyan!40] (0:1.5) -- (72:1.5) -- (-72:1.5);
\draw[cyan!40,fill=cyan!40,line width=0] (1.5,0) circle (.3);
\draw[black!100,line width=2pt,dotted] (72:1.5) -- (-72:1.5);
\foreach \i in {1,...,5} 
{
\draw[fill=black!100,line width=1] (\i*72:1.5) circle (.1);
\draw[black!100,line width=1pt] (\i*72:1.5) -- (72+\i*72:1.5);
}

 \begin{scope}[shift={(9,0)}]

   \draw[cyan!40,fill=cyan!40,line width=16, line cap=round] (72:1.5) -- (144:1.5);
\draw[cyan!40,fill=cyan!40,line width=16, line cap=round] (288:1.5) -- (216:1.5);
\draw[cyan!40,fill=cyan!40,line width=16, line cap=round] (-.5,0) -- (.5,0);
\draw[cyan!40,fill=cyan!40,line width=0] (1.5,0) circle (.3);
\draw[fill=black!100,line width=1] (-.5,0) circle (.1);
\draw[fill=black!100,line width=1] (.5,0) circle (.1);
\draw[black!100,line width=2pt,dotted] (-.5,0) -- (.5,0);
\foreach \i in {1,...,5} 
{
\draw[fill=black!100,line width=1] (\i*72:1.5) circle (.1);
\draw[black!100,line width=1pt] (\i*72:1.5) -- (72+\i*72:1.5);
}

 \end{scope}
 
\begin{scope}[shift={(4.5,0)}]
\draw[cyan!40,fill=cyan!40,line width=16, line cap=round] (72:1.5) -- (144:1.5);
\draw[cyan!40,fill=cyan!40,line width=16, line cap=round] (288:1.5) -- (216:1.5);
\draw[cyan!40,fill=cyan!40,line width=16, line cap=round] (0,0) -- (1.5,0);
\draw[fill=black!100,line width=1] (0,0) circle (.1);
\draw[black!100,line width=2pt,dotted] (1.5,0) -- (0,0);
\foreach \i in {1,...,5} 
{
\draw[fill=black!100,line width=1] (\i*72:1.5) circle (.1);
\draw[black!100,line width=1pt] (\i*72:1.5) -- (72+\i*72:1.5);
}
  
 \end{scope}
 
\end{tikzpicture}
\caption{Illustration of the three ways in which an edge can be added to a graph $G\sim C_5 + E_{n-5}$, up to isomorphism, given large enough $n$. The blue shading indicates a clique cover of cardinality n-3 in each case.}
\label{fig:c5+edge}
\end{figure}

We now present a general graph that preserves the ``nice'' properties of the 5-cycle in the form of a slight modification of the complete bipartite graph. 

For any integer $a\ge 2$ let $K^*_{a,a}$ be the complete bipartite graph $K_{a,a}$ with one subdivided edge. We shall label the vertices $x_1,\ldots, x_{a}$ and $y_1,\ldots y_{a}$ where all vertices $x_i$ are in one class of the partition and all vertices $y_i$ are in the other.  We will subdivide the edge $x_1y_1$, that is, it will be replaced by edges $x_1v_0$ and $y_1v_0$ where $v_0$ is an additional vertex, of degree 2. Note that $K^*_{2,2}\sim C_5$.

\begin{lemma}\label{lem:gnKaaupp}
For $a \ge 3$ and $s$ positive integers, we have $\gn_s(K^*_{a,a})\leq a+\frac{2}{3}$.
\end{lemma}

\begin{proof}
Let $\mathcal{P}$ be an optimal guessing protocol for the guessing game on $G$ with $s$ colors. Let $Z$ be a random variable that picks uniformly from all fixed points of $\mathcal{P}$. For brevity, we will use the notation $X_{[k:l]}=X_k,\dots,X_l$. Then, the random variable $Z=(V_0, X_{[1:a]},Y_{[1:a]})$ where $X_i$ refers to the color assigned to vertex $x_i$, with a similar definition for $V_0$ and $Y_i$. 

Since $x_{[1:a]}$ and $y_{[1:a]}$ are both independent sets, by Lemma~\ref{lem:entpop} and~\ref{lem:gnent}, we have

\begin{equation*}
    \gn (K^*_{a,a})=H(Z)=H(V_0, X_1, Y_{[1:a]})=H(V_0, X_{[1:a]}, Y_1)= H(V_0, X_{[1:a]}).
\end{equation*}
Therefore,
\begin{align}
    3\cdot & \gn(K^*_{a,a}) = 3\cdot H(Z)\notag \\
                 &=  H(V_0, X_{[1:a]}, Y_1) + H(V_0, X_1, Y_{[1:a]}) + H(V_0, X_{[1:a]})\notag \
                 \\
                 &\leq H(V_0, X_{[1:a]}, Y_1) + H(V_0, X_1, Y_{[1:a]}) + H(V_0) + H(X_{[1:a]})
                 \label{en:st:popv0}
                 \\\
                 &\leq H(V_0, X_{[1:a]}, Y_1) + H(V_0, X_1, Y_{[2:a]}) + H(V_0, Y_1) + H(X_{[1:a]})\label{en:st:shgy1}\\\
                 &= H(V_0, X_{[1:a]}, Y_1) + H(V_0, Y_1) + H(V_0, X_1, Y_{[2:a]}) + H(X_{[1:a]})\notag\\
                 &\leq H(V_0, X_{[2:a]}, Y_1) + H(V_0, Y_1, X_1) + H(V_0, X_1, Y_{[2:a]}) + H(X_{[1:a]})\label{en:st:shgx1}\\\
                 &= H(V_0, X_{[2:a]}) + H(Y_1, X_1) + H(V_0, Y_{[2:a]}) + H(X_{[1:a]})\label{en:st:popin}\\\
                 &\leq a+2+a +a \label{en:st:vxbnd}\\\
                 &= 3a+2\notag
\end{align}
With inequalities~(\ref{en:st:popv0}), (\ref{en:st:shgy1}), and (\ref{en:st:shgx1}) relying on Lemma \ref{lem:shannon}, inequality (\ref{en:st:popin}) relying on Lemma~\ref{lem:entpop} and inequality~(\ref{en:st:vxbnd}) relying on Lemma~\ref{lem:entleq1}.
\end{proof}

\begin{lemma}\label{lem:Kaaadde}
Let $G\sim K^*_{a,a}+E_{n-2a-1}$ be a graph on $n\geq 2a-1$ vertices. Then for any $e\in \overline{E(G)}$, we have $\gn(G+e)= a+1$.
\end{lemma}

\begin{proof}
It is not difficult to verify that, up to isomorphism, there are at most eight ways to add an edge to $G$. In each of these cases, $G+e$ has a clique cover of cardinality $n-a-1$. We provide proof by illustration in Figure~\ref{fig:satex}. It is also not difficult to see that $\alpha (G+e)=n-a-1$ in each of these cases, giving us the equality.
\end{proof}

\begin{figure}[h]
\centering
  \begin{tikzpicture}[scale=1]
  \draw[cyan!40,fill=cyan!40,line width=2] (0,2) -- (0,3) -- (2,3);
  \draw[cyan!40,fill=cyan!40,line width=16, line cap=round] (0,3) -- (2,3);
   \draw[cyan!40,fill=cyan!40,line width=16, line cap=round] (0,3) -- (0,2);
    \draw[cyan!40,fill=cyan!40,line width=16, line cap=round] (2,3) -- (0,2);
    \draw[cyan!40,fill=cyan!40,line width=16, line cap=round] (2,2) -- (0,1);
\draw[cyan!40,fill=cyan!40,line width=16, line cap=round] (0,0) -- (2,1);
\draw[cyan!40,fill=cyan!40,line width=16, line cap=round] (1,-.5) -- (2,0);
 \draw[fill=black!100,line width=1] (1,-.5) circle (.1);
\foreach \i in {1,...,4} 
{
 \draw[cyan!40,fill=cyan!40,line width=0] (0,-1+\i) circle (.15);
  \draw[cyan!40,fill=cyan!40,line width=0] (2,-1+\i) circle (.15);
\draw[fill=black!100,line width=1] (0,-1+\i) circle (.1);
\draw[fill=black!100,line width=1] (2,-1+\i) circle (.1);
}
\foreach \i in {2,...,4}
{\foreach \j in {1,...,4}
{\draw[black!100,line width=1pt] (0,-1+\i) -- (2,-1+\j);}
}
\foreach \j in {2,...,4}
{\draw[black!100,line width=1pt] (0,0) -- (2,-1+\j);}
 \draw[black!100,line width=2pt,dotted] (0,2) -- (0,3);
\draw[black!100,line width=1pt] (0,0) -- (1,-.5);
\draw[black!100,line width=1pt] (2,0) -- (1,-.5);

\begin{scope}[shift={(4,0)}]
\draw[cyan!40,fill=cyan!40,line width=2] (0,0) -- (2,0) -- (1,-.5);
  \draw[cyan!40,fill=cyan!40,line width=16, line cap=round] (0,0) -- (1,-.5);
   \draw[cyan!40,fill=cyan!40,line width=16, line cap=round] (0,0) -- (2,0);
    \draw[cyan!40,fill=cyan!40,line width=16, line cap=round] (2,0) -- (1,-.5);
\draw[cyan!40,fill=cyan!40,line width=16, line cap=round] (0,1) -- (2,1);
\draw[cyan!40,fill=cyan!40,line width=16, line cap=round] (0,2) -- (2,2);
\draw[cyan!40,fill=cyan!40,line width=16, line cap=round] (0,3) -- (2,3);
 \draw[fill=black!100,line width=1] (1,-.5) circle (.1);
\foreach \i in {1,...,4} 
{
 \draw[cyan!40,fill=cyan!40,line width=0] (0,-1+\i) circle (.15);
  \draw[cyan!40,fill=cyan!40,line width=0] (2,-1+\i) circle (.15);
\draw[fill=black!100,line width=1] (0,-1+\i) circle (.1);
\draw[fill=black!100,line width=1] (2,-1+\i) circle (.1);
}
\foreach \i in {2,...,4}
{\foreach \j in {1,...,4}
{\draw[black!100,line width=1pt] (0,-1+\i) -- (2,-1+\j);}
}
\foreach \j in {2,...,4}
{\draw[black!100,line width=1pt] (0,0) -- (2,-1+\j);}
 \draw[black!100,line width=2pt,dotted] (0,0) -- (2,0);
\draw[black!100,line width=1pt] (0,0) -- (1,-.5);
\draw[black!100,line width=1pt] (2,0) -- (1,-.5);
\end{scope}

\begin{scope}[shift={(12,0)}]
 \draw[cyan!40,fill=cyan!40,line width=2] (0,0) -- (0,1) -- (2,1);
  \draw[cyan!40,fill=cyan!40,line width=16, line cap=round] (0,3) -- (2,3);
   \draw[cyan!40,fill=cyan!40,line width=16, line cap=round] (0,1) -- (0,0);
    \draw[cyan!40,fill=cyan!40,line width=16, line cap=round] (2,3) -- (0,3);
\draw[cyan!40,fill=cyan!40,line width=16, line cap=round] (0,1) -- (2,1);
\draw[cyan!40,fill=cyan!40,line width=16, line cap=round] (0,0) -- (2,1);
    \draw[cyan!40,fill=cyan!40,line width=16, line cap=round] (2,2) -- (0,2);
\draw[cyan!40,fill=cyan!40,line width=16, line cap=round] (1,-.5) -- (2,0);
 \draw[fill=black!100,line width=1] (1,-.5) circle (.1);
\foreach \i in {1,...,4} 
{
 \draw[cyan!40,fill=cyan!40,line width=0] (0,-1+\i) circle (.15);
  \draw[cyan!40,fill=cyan!40,line width=0] (2,-1+\i) circle (.15);
\draw[fill=black!100,line width=1] (0,-1+\i) circle (.1);
\draw[fill=black!100,line width=1] (2,-1+\i) circle (.1);
}
\foreach \i in {2,...,4}
{\foreach \j in {1,...,4}
{\draw[black!100,line width=1pt] (0,-1+\i) -- (2,-1+\j);}
}
\foreach \j in {2,...,4}
{\draw[black!100,line width=1pt] (0,0) -- (2,-1+\j);}
 \draw[black!100,line width=2pt,dotted] (0,0) -- (0,1);
\draw[black!100,line width=1pt] (0,0) -- (1,-.5);
\draw[black!100,line width=1pt] (2,0) -- (1,-.5);
\end{scope}

 \begin{scope}[shift={(0,-5)}]
  \draw[cyan!40,fill=cyan!40,line width=16, line cap=round] (0,0) -- (1,-.5);
   \draw[cyan!40,fill=cyan!40,line width=16, line cap=round] (0,1) -- (2,0);
\draw[cyan!40,fill=cyan!40,line width=16, line cap=round] (2,1) -- (2.5,1);
\draw[cyan!40,fill=cyan!40,line width=16, line cap=round] (0,1) -- (2,0);
\draw[cyan!40,fill=cyan!40,line width=16, line cap=round] (0,2) -- (2,2);
\draw[cyan!40,fill=cyan!40,line width=16, line cap=round] (0,3) -- (2,3);
 \draw[fill=black!100,line width=1] (1,-.5) circle (.1);
\foreach \i in {1,...,4} 
{
 \draw[cyan!40,fill=cyan!40,line width=0] (0,-1+\i) circle (.15);
  \draw[cyan!40,fill=cyan!40,line width=0] (2,-1+\i) circle (.15);
\draw[fill=black!100,line width=1] (0,-1+\i) circle (.1);
\draw[fill=black!100,line width=1] (2,-1+\i) circle (.1);
}
\foreach \i in {2,...,4}
{\foreach \j in {1,...,4}
{\draw[black!100,line width=1pt] (0,-1+\i) -- (2,-1+\j);}
}
\foreach \j in {2,...,4}
{\draw[black!100,line width=1pt] (0,0) -- (2,-1+\j);}
 \draw[black!100,line width=2pt,dotted] (2,1) -- (2.5,1);
\draw[black!100,line width=1pt] (0,0) -- (1,-.5);
\draw[black!100,line width=1pt] (2,0) -- (1,-.5);
\draw[fill=black!100,line width=1] (2.5,1) circle (.1);
\end{scope}  

\begin{scope}[shift={(4,-5)}]
  \draw[cyan!40,fill=cyan!40,line width=16, line cap=round] (0,0) -- (1,-.5);
   \draw[cyan!40,fill=cyan!40,line width=16, line cap=round] (0,1) -- (2,1);
\draw[cyan!40,fill=cyan!40,line width=16, line cap=round] (2,0) -- (2.5,0);
\draw[cyan!40,fill=cyan!40,line width=16, line cap=round] (0,2) -- (2,2);
\draw[cyan!40,fill=cyan!40,line width=16, line cap=round] (0,3) -- (2,3);
 \draw[fill=black!100,line width=1] (1,-.5) circle (.1);
\foreach \i in {1,...,4} 
{
 \draw[cyan!40,fill=cyan!40,line width=0] (0,-1+\i) circle (.15);
  \draw[cyan!40,fill=cyan!40,line width=0] (2,-1+\i) circle (.15);
\draw[fill=black!100,line width=1] (0,-1+\i) circle (.1);
\draw[fill=black!100,line width=1] (2,-1+\i) circle (.1);
}
\foreach \i in {2,...,4}
{\foreach \j in {1,...,4}
{\draw[black!100,line width=1pt] (0,-1+\i) -- (2,-1+\j);}
}
\foreach \j in {2,...,4}
{\draw[black!100,line width=1pt] (0,0) -- (2,-1+\j);}
 \draw[black!100,line width=2pt,dotted] (2,0) -- (2.5,0);
\draw[black!100,line width=1pt] (0,0) -- (1,-.5);
\draw[black!100,line width=1pt] (2,0) -- (1,-.5);
\draw[fill=black!100,line width=1] (2.5,0) circle (.1);
\end{scope}

 \begin{scope}[shift={(8,0)}]
  \draw[cyan!40,fill=cyan!40,line width=16, line cap=round] (0,0) -- (1,-.5);
   \draw[cyan!40,fill=cyan!40,line width=16, line cap=round] (0,1) -- (2,0);
\draw[cyan!40,fill=cyan!40,line width=16, line cap=round] (0,1) -- (2,0);
\draw[cyan!40,fill=cyan!40,line width=16, line cap=round] (0,2) -- (2,2);
\draw[cyan!40,fill=cyan!40,line width=16, line cap=round] (0,3) -- (2,3);
\draw[cyan!40,fill=cyan!40,line width=16, line cap=round] (0,0) -- (2,1);
\draw[cyan!40,fill=cyan!40,line width=16, line cap=round] (1,-.5) -- (2,1);
\draw[cyan!40,fill=cyan!40,line width=2] (0,0) -- (2,1) -- (1,-.5);
 \draw[fill=black!100,line width=1] (1,-.5) circle (.1);
\foreach \i in {1,...,4} 
{
 \draw[cyan!40,fill=cyan!40,line width=0] (0,-1+\i) circle (.15);
  \draw[cyan!40,fill=cyan!40,line width=0] (2,-1+\i) circle (.15);
  \draw[cyan!40,fill=cyan!90] (1.6,-.1) -- (1.9,.4) -- (1,.8) -- (.4,.5);
\draw[fill=black!100,line width=1] (0,-1+\i) circle (.1);
\draw[fill=black!100,line width=1] (2,-1+\i) circle (.1);
}
\foreach \i in {2,...,4}
{\foreach \j in {1,...,4}
{\draw[black!100,line width=1pt] (0,-1+\i) -- (2,-1+\j);}
}
\foreach \j in {2,...,4}
{\draw[black!100,line width=1pt] (0,0) -- (2,-1+\j);}
 \draw[black!100,line width=2pt,dotted] (2,1) -- (1,-.5);
\draw[black!100,line width=1pt] (0,0) -- (1,-.5);
\draw[black!100,line width=1pt] (2,0) -- (1,-.5);
\end{scope}

 \begin{scope}[shift={(12,-5)}]
  \draw[cyan!40,fill=cyan!40,line width=16, line cap=round] (1,-.5) -- (1,-1);
\draw[cyan!40,fill=cyan!40,line width=16, line cap=round] (0,1) -- (2,0);
\draw[cyan!40,fill=cyan!40,line width=16, line cap=round] (2,1) -- (0,0);
\draw[cyan!40,fill=cyan!40,line width=16, line cap=round] (0,2) -- (2,2);
\draw[cyan!40,fill=cyan!40,line width=16, line cap=round] (0,3) -- (2,3);
\draw[cyan!40,fill=cyan!40,line width=2] (0,0) -- (1,-.5);
 \draw[fill=black!100,line width=1] (1,-.5) circle (.1);
\foreach \i in {1,...,4} 
{
 \draw[cyan!40,fill=cyan!40,line width=0] (0,-1+\i) circle (.15);
  \draw[cyan!40,fill=cyan!40,line width=0] (2,-1+\i) circle (.15);
\draw[cyan!40,fill=cyan!90] (1,.2) -- (1.6,.5) -- (1,.8) -- (.4,.5);
\draw[fill=black!100,line width=1] (0,-1+\i) circle (.1);
\draw[fill=black!100,line width=1] (2,-1+\i) circle (.1);
}
\draw[fill=black!100,line width=1] (1,-1) circle (.1);
 \draw[black!100,line width=2pt,dotted] (1,-.5) -- (1,-1);
\foreach \i in {2,...,4}
{\foreach \j in {1,...,4}
{\draw[black!100,line width=1pt] (0,-1+\i) -- (2,-1+\j);}
}
\foreach \j in {2,...,4}
{\draw[black!100,line width=1pt] (0,0) -- (2,-1+\j);}
\draw[black!100,line width=1pt] (0,0) -- (1,-.5);
\draw[black!100,line width=1pt] (2,0) -- (1,-.5);
\end{scope}  

\begin{scope}[shift={(8,-5)}]
  \draw[cyan!40,fill=cyan!40,line width=16, line cap=round] (0,0) -- (1,-.5);
\draw[cyan!40,fill=cyan!40,line width=16, line cap=round] (0,1) -- (2,1);
\draw[cyan!40,fill=cyan!40,line width=16, line cap=round] (2.7,1) -- (2.7,2);
\draw[cyan!40,fill=cyan!40,line width=16, line cap=round] (0,2) -- (2,2);
\draw[cyan!40,fill=cyan!40,line width=16, line cap=round] (0,3) -- (2,3);
\draw[cyan!40,fill=cyan!40,line width=0] (2,0) circle (.3);
\draw[cyan!40,fill=cyan!40,line width=2] (0,0) -- (1,-.5);
 \draw[fill=black!100,line width=1] (1,-.5) circle (.1);
\foreach \i in {1,...,4} 
{
 \draw[cyan!40,fill=cyan!40,line width=0] (0,-1+\i) circle (.15);
  \draw[cyan!40,fill=cyan!40,line width=0] (2,-1+\i) circle (.15);
\draw[fill=black!100,line width=1] (0,-1+\i) circle (.1);
\draw[fill=black!100,line width=1] (2,-1+\i) circle (.1);
}
\draw[fill=black!100,line width=1] (2.7,1) circle (.1);
\draw[fill=black!100,line width=1] (2.7,2) circle (.1);
 \draw[black!100,line width=2pt,dotted] (2.7,1) -- (2.7,2);
\foreach \i in {2,...,4}
{\foreach \j in {1,...,4}
{\draw[black!100,line width=1pt] (0,-1+\i) -- (2,-1+\j);}
}
\foreach \j in {2,...,4}
{\draw[black!100,line width=1pt] (0,0) -- (2,-1+\j);}
\draw[black!100,line width=1pt] (0,0) -- (1,-.5);
\draw[black!100,line width=1pt] (2,0) -- (1,-.5);
\end{scope}  
 
\end{tikzpicture}
\caption{Illustration of the eight ways in which an edge can be added to a graph $G\sim K^*_{4,4} + E_{n-9}$, up to isomorphism, given large enough $n$. The blue shading indicates a clique cover of cardinality $n-5$ in each case.}
\label{fig:satex}
\end{figure}
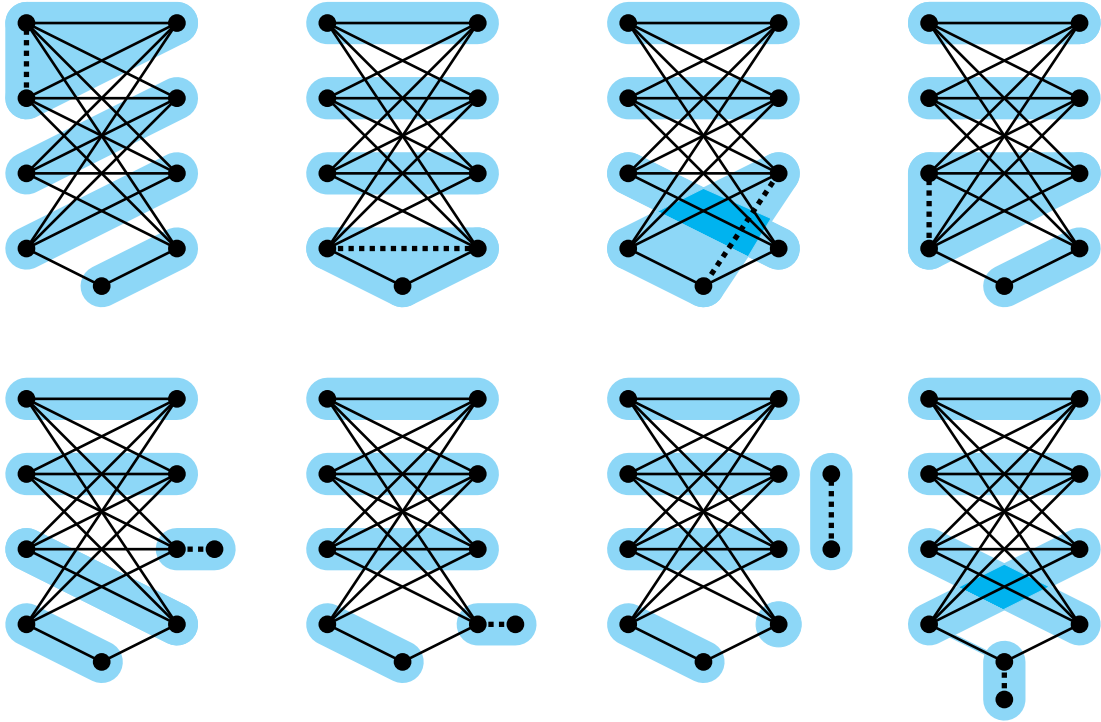

\begin{theorem}
\label{sat:bound}
Let $a\leq 2$ and $s$ be positive integers.
For $n\geq 2a+1$, we have $\sat(n,\gn_s\geq a+1)\leq a^2+1$.
\end{theorem}

\begin{proof}
Consider the graph $G \sim K^*_{a,a}+E_{n-2a-1}$. By Lemma~\ref{lem:gnKaaupp}, we have that $\gn_s (G)\leq \gn(G)  \leq a+\tfrac{2}{3}<a+1.$ By Lemma~\ref{lem:Kaaadde}, we see that for any $e \in \overline{E(G)}$, we have $\gn_s(G+e)=a+1$.
\end{proof}

It should be noted that unlike the extremal number this constrution does not necessarily work for \textit{all} guessing numbers.  For a guessing number more than $\frac{1}{3}$ below an integer it is not clear if our construction has the correct guessing number to be saturated.

As we did in Corollary~\ref{cor:monexgn}, we can consider the monotonicity of the saturation function. In the case of forbidden subgraphs, it is known that the saturation function fails to have any of the monotonicity properties listed in Lemma~\ref{lem:monF}(i)-(iii). For the case of guessing numbers, saturation fails to have any of the monotonicity properties listed in Corollary~\ref{cor:monexgn}. As a counterexample to $\sat (n,\gn_s \geq a)\leq \sat (n,\gn_s \geq b)$ (the analogue of Corollary~\ref{cor:monexgn}(i)), we have seen that, when $n\geq 7$, we have that $\sat (n,\gn \geq 2) > \sat (n,\gn \geq 3)$. As a counterexample to $\sat (n,\gn_s \geq a)\leq \sat (n+1,\gn_s \geq a)$ (the analogue of Corollary~\ref{cor:monexgn}(ii)), we have that $\sat (6,\gn_s \geq 4)=10$ and $\sat (7,\gn_s \geq 4)=9$. (Verified by computer.)

\section{Iterative Construction of Saturated Graphs}\label{sec:spect}
 
In 1986, K\'{a}szonyi and Tuza provided a general saturated graph construction which proves that $\sat(n,\mathcal{F})=O(n)$ for every family of graphs $\mathcal{F}$~\cite{kaszonyi_saturated_1986}. Their construction is based on the following observation. For $\mathcal{F}$ a family of graphs, let 
\[ \mathcal{F}'=\{ F-v\; | \; F\in \mathcal{F},\; v\in V(F) \}.  \]
 
 \begin{lemma}\label{lem:vxadd}
For any graph $G$ and vertex $v \in G$, and for any number of colors $s$, we have 
\[ \gn (G-v,s) \leq \gn (G,s) \leq \gn (G-v,s)+1.   \]
Furthermore, if $N[w] \subseteq N(v)$ for any $w \in V(g) \setminus \{v\}$, then 
\[ \gn (G,s) = \gn (G-v,s)+1. \]
\end{lemma}

\begin{proof}
The lower bound follows from Lemma~\ref{lem:subgraphadd}. To prove the upper bound, suppose that we have a graph $G$ and vertex $v \in V(G)$ such that $\gn (G,s) > \gn (G-v,s)+1$, for the sake of contradiction. Then there exists a strategy $\mathcal{P}$ such that 
\[\fix (\mathcal{P}) >s^{\gn (G-v,s)+1}  . \]
Since there are only $s$ colors, this suggests that there is a color $j$ such that vertex $v$ has color $j$ in more than $s^{\gn (G-v,s)}$ of the fixed points of $\mathcal{P}$. 
 However, if we take the set of colorings given by these fixed points and restrict them to the set $V(G) \setminus \{v\}$, we obtain a strategy for $G-v$ with more than $s^{\gn (G-v,s)}$ the fixed points. This is in contradiction with the definition of $\gn (G-v,s)$.

Now suppose that there exists a $w \in V(g) \setminus \{v\}$ such that $N[w] \subseteq N(v)$. We will prove that $\gn (G,s) \geq \gn (G-v,s)+1$ by extending an optimal strategy on $G-v$ to a strategy on $v$. Informally, one can think of the new strategy as following the old strategy, except that we pretend that $v$ and $w$ are a single vertex with color $c(v)+c(w) \pmod{s}$.
Formally, label the vertices of $G$ as 
\[V(G)=\{v_1,v_2,\dots,v_{n-2},v_{n-1}=w,v_n=v \}.  \]
Let $\mathcal{P'}=\{  f_i' \}_1^{n-1}$ be an optimal strategy on $G-v$. We construct $\mathcal{P}=\{  f_i \}_1^{n}$, a strategy on $G$, as follows. For a given coloring $c$ of $V(G)$, let $c'(v_i)=c(v_i)$ for $1 \leq i \leq n-2$, and let $c'(w)=c(w)+c(v) \pmod{s}$. Then, we let $f_i(c)=f_i'(c')$ for $1 \leq i \leq n-2$. Furthermore, we let $f_{n-1}(c)=f_i'(c')-c(v)$ and $f_n(c)=f_i'(c')-c(w)$. This strategy gives 
\[ |\fix (\mathcal{P}) |=s \cdot  |\fix (\mathcal{P'}) | , \]
because for every $c' \in \fix (\mathcal{P'}) $ in which $c' (w)=j$, there are exactly $j$ colorings in $\fix (\mathcal{P'})$. We find these colorings by letting $c(v_i)=c'(v_i)$ for $1 \leq i 
\leq n-2$ and letting $c(w) \in \{ 0,\dots,s-1 \}$ with $c(v)=j-c(w) \pmod{s}$.
\end{proof}
 
 \begin{corollary}
\label{cor:vxaddsat}
Let $G$ be a graph with a dominating vertex $v$. Then 
\[ G-v\in \Sat (n, \gn_s \geq a)\;\; \Leftrightarrow \;\; G\in \Sat (n, \gn_s \geq a+1). \]
 \end{corollary}
 \begin{proof}
 Suppose that $G-v\in \Sat (n, \gn_s \geq a)$. Then $\gn_s(G-v)<a$ and for any $e \in \overline{E(G-v)}$ we have $\gn_s(G-v+e)\geq a$. Since $v$ is a dominating vertex, we have $\overline{E(G-v)}=\overline{E(G)}$. Therefore, by Lemma~\ref{lem:vxadd}, the graph $G$ has the property that $\gn_s(G)<a+1$ and for any $e \in \overline{E(G)}$ we have $\gn_s(G+e)\geq a+1$. Therefore, $G\in \Sat (n, \gn_s \geq a+1)$. The other direction of the biconditional statement follows in a very similar manner.
 \end{proof}

 \begin{corollary}\label{cor:spectrum}
 When $a$ is a positive integer and $n\geq 2a+1$, there exist graphs on $n$ vertices that are $(\gn\geq a+1)$-saturated on any number of edges in the set
\[\left. \left\{ \binom{b}{2}\cdot b \cdot (n-b)  + (a-b)^2+1 \;\; \right|\;\; b \in \{ 0,\dots,a\} \right\}. \] 
 \end{corollary}
 \begin{proof}
 By Theorem~\ref{sat:bound} there exists a graph on $n-b$ vertices and $(a-b)^2+1$ edges that is $(\gn\geq a-b+1)$-saturated. Then, by repeated use of Corollary~\ref{cor:vxaddsat}, we add $b$ dominating vertices to this graph to obtain a graph on $n$ vertices and $ \binom{b}{2}\cdot b \cdot (n-b)  + (a-b)^2+1 $ edges that is $(\gn\geq a+1)$-saturated.
 \end{proof}
 
 We note that the construction described in Corollary~\ref{cor:spectrum} encompasses both the construction used to find the exact extremal number in Theorem~\ref{thm:extremal} (by setting $b=a$) as well as the construction used to find an upper bound on the saturation number in Theorem~\ref{sat:bound} (by setting $b=0$). 
 
 \section{Guessing Number and Forbidden Subgraphs}\label{sec:Fgneq}
We conclude with a result that shows some of the relationship between the bounded guessing number property and forbidden subgraphs.
 
\begin{lemma}\label{lem:Fgneq}
For every $s \in \mathbb{N}$, $a \in \mathbb{R}$, there exists a unique finite family of minimal forbidden subgraphs $\mathcal{F}_{s,a}$ such that, for any graph $G$,
\[ \gn_s(G)< a \;\; \Leftrightarrow \;\; G \mbox{ is }\mathcal{F}_{s,a}\mbox{-free}. \]
\end{lemma}
\begin{proof}
First of all, note that the properties $\gn_s(G)< a$ and $\gn(G)< a$ are preserved by the removal of edges or vertices from $G$, and are therefore characterized by forbidden subgraphs. All that remains to be shown is that there are only finitely many graphs in the family $\mathcal{F}_{s,a}$. We will do this by showing that 
\begin{equation}\label{eq:Forderbound}
    |F| \leq s^{s^{2(a+1)}}+2(a+1), 
\end{equation} 
for any $F \in \mathcal{F}_{s,a} $. Suppose that we have some minimal forbidden subgraph $F \in \mathcal{F}_{s,a}$. Let $n=|F|$. By Lemma~\ref{lem:vxadd}, we must have that 
\[ a \leq  \gn_s (F) < a+1. \]
Otherwise, we would have $\gn_s(F-v)\geq a$ for any $v \in V(F)$ which contradicts minimality of $F$. Let $\alpha ' (F)$ indicate the cardinality of a maximum matching in $F$. By Lemma~\ref{lem:cpalpha}, we must have that
\[ \alpha' = \alpha'(F) <a+1\; \mbox{ and } \; \alpha=\alpha(F)<n-a.\]
It is shown in~\cite{willis2011bounds} that for any graph on $n$ vertices with independence and edge independence numbers $\alpha$ and $\alpha'$, respectively, we have
\begin{equation}\label{eq:alphasum}
     n \leq 2\alpha' + \alpha  .\end{equation}
Suppose, for the sake of contradiction, that $\alpha  > s^{s^{n-\alpha}}$.
Let $A$ be a maximum independent set of $F$. Each vertex $v_i$ in $A$ has at most $s^{s^{n-\alpha}}$ possible strategies $f_i$, since $|N(v_i)|\leq n-\alpha$. Let $\mathcal{P}=\{f_i\}_1^{n}$ be an optimal strategy on $F$. Then $\alpha  > s^{s^{n-\alpha}}$ implies, by the Pigeonhole Principle, that there are two vertices $v_i$ and $v_j$, both in $A$, such that $f_i=f_j$. If $N(v_i)=N(v_j)$, then clearly the graph $F-v_i$ has the same guessing number as $F$, which contradicts the minimality of $F$. Otherwise, without loss of generality, there exists a vertex $w \in N(v_i) \setminus N(v_j)$. Since the strategy of $v_j$ does not depend on the color of $w$, and $f_i=f_j$, we have that the strategy of $v_i$ does not depend on the color of $w$. Therefore, the strategy is $\mathcal{P}$ is valid on the graph $F-v_ix$, which contradicts the minimality of $F$. Therefore, we must have that $\alpha  \leq s^{s^{n-\alpha}}$. Combining this with Inequality~(\ref{eq:alphasum}) gives the result in Inequality~(\ref{eq:Forderbound}).
\end{proof}

Note that we can let $\mathcal{F}_{a}$ be such that for any graph $G$
\[ \gn(G)< a \;\; \Leftrightarrow \;\; G \mbox{ is }\mathcal{F}_{a}\mbox{-free}. \]
Then, $\mathcal{F}_{a}=\mathcal{F}_{2,a} \cup \mathcal{F}_{3,a} \cup \dots$ and it is not clear whether this family needs to be finite. We leave this to the reader as an open question.

\bibliographystyle{plain}
\bibliography{references.bib}
\end{document}